\documentclass[a4paper,11pt]{amsart}

\usepackage[left=1.4in,right=1.4in,top=1.1in,bottom=1.1in,includehead,includefoot]{geometry}
\usepackage[OT4]{fontenc}
\usepackage{amsthm,amsmath,amssymb}
\usepackage{enumitem}
\usepackage[pdftitle={An extremal problem on crossing vectors},
            pdfauthor={M. Lason, P. Micek, N. Streib, W. T. Trotter, B. Walczak}]{hyperref}

\newenvironment{enumeratei}{\begin{enumerate}[label={\upshape(\roman*)}, noitemsep, topsep=2pt plus 2pt, leftmargin=*, widest=iii]}{\end{enumerate}}

\newtheorem{theorem}{Theorem}
\newtheorem{conjecture}[theorem]{Conjecture}
\newtheorem{proposition}[theorem]{Proposition}
\newtheorem{claim}[theorem]{Claim}

\newcommand{\cgA}{\mathcal{A}}
\newcommand{\cgB}{\mathcal{B}}
\newcommand{\cgF}{\mathcal{F}}
\newcommand{\cgM}{\mathcal{M}}
\newcommand{\cgP}{\mathcal{P}}
\newcommand{\ZZ}{\mathbb{Z}}

\makeatletter
\def\pmod#1{\ ({\operator@font mod}\,\,#1)}
\makeatother

\let\leq\leqslant
\let\geq\geqslant
\let\setminus\smallsetminus

\makeatletter
\let\old@setaddresses\@setaddresses
\def\@setaddresses{\bgroup\parindent 0pt\let\scshape\relax\old@setaddresses\egroup}
\makeatother

\title{An extremal problem on crossing vectors}

\author[M.~Laso\protect\'n]{Micha\l{}~Laso\'n}
\author[P.~Micek]{Piotr~Micek}
\author[N.~Streib]{Noah~Streib}
\author[W.~T.~Trotter]{William~T.~Trotter}
\author[B.~Walczak]{Bartosz~Walczak}

\address[Micha\l{} Laso\'n]{Theoretical Computer Science Department, Faculty of Mathematics and Computer Science, Jagiellonian University, Krak\'ow, Poland; Institute of Mathematics of the Polish Academy of Sciences, Warsaw, Poland; \'Ecole Polytechnique F\'ed\'erale de Lausanne, Switzerland}
\email{michalason@gmail.com}

\address[Piotr Micek]{Theoretical Computer Science Department, Faculty of Mathematics and Computer Science, Jagiellonian University, Krak\'ow, Poland}
\email{micek@tcs.uj.edu.pl}

\address[Noah Streib, William T. Trotter]{School of Mathematics, Georgia Institute of Technology, Atlanta, GA~30332, USA}
\email{nstreib3@math.gatech.edu, trotter@math.gatech.edu}

\address[Bartosz Walczak]{Theoretical Computer Science Department, Faculty of Mathematics and Computer Science, Jagiellonian University, Krak\'ow, Poland; \'Ecole Polytechnique F\'ed\'erale de Lausanne, Switzerland}
\email{walczak@tcs.uj.edu.pl}

\thanks{Journal version of this paper appeared in \emph{J.\ Combin.\ Theory Ser.\ A}, 128:41--55, 2014.}
\thanks{Micha\l{} Laso\'n was supported by Polish National Science Centre grant N~N206 568240 and by Swiss National Science Foundation grants 200020-144531 and 200021-137574.}
\thanks{Bartosz Walczak was supported by Swiss National Science Foundation grant 200020-144531.}

\begin{document}
\baselineskip 14pt

\begin{abstract}
For positive integers $w$ and $k$, two vectors $A$ and $B$ from $\mathbb{Z}^w$ are called \emph{$k$-crossing} if there are two coordinates $i$ and $j$ such that $A[i]-B[i]\geq k$ and $B[j]-A[j]\geq k$.
What is the maximum size of a family of pairwise $1$-crossing and pairwise non-$k$-crossing vectors in $\mathbb{Z}^w$?
We state a conjecture that the answer is $k^{w-1}$.
We prove the conjecture for $w\leq 3$ and provide weaker upper bounds for $w\geq 4$.
Also, for all $k$ and $w$, we construct several quite different examples of families of desired size $k^{w-1}$.
This research is motivated by a natural question concerning the width of the lattice of maximum antichains of a partially ordered~set.
\end{abstract}

\maketitle

\section{Introduction}

We deal with vectors in $\ZZ^w$, which we call just \emph{vectors}.
The $i$th coordinate of a vector $A\in\ZZ^w$ is denoted by $A[i]$, for $1\leq i\leq w$.
The product ordering on $\ZZ^w$ is defined by setting $A\leq B$ for $A,B\in\ZZ^w$ whenever $A[i]\leq B[i]$ for every coordinate~$i$.
When $k\geq 1$, we say that vectors $A$ and $B$ from $\ZZ^w$ are \emph{$k$-crossing} if there are coordinates $i$ and $j$ for which $A[i]-B[i]\geq k$ and $B[j]-A[j]\geq k$.
Thus $\cgA$ is an antichain in $\ZZ^w$ if and only if any two distinct vectors from $\cgA$ are $1$-crossing.
A family of vectors in $\ZZ^w$ is \emph{$k$-crossing-free} if it contains no two $k$-crossing vectors.

For positive integers $k$ and $w$, let $f(k,w)$ denote the maximum size of a subset of $\ZZ^w$ with any two vectors being $1$-crossing but not $k$-crossing.
In other words, $f(k,w)$ is the maximum size of a $k$-crossing-free antichain in $\ZZ^w$.
Note that an antichain of vectors in $\ZZ^w$ with $w\geq 2$ without the restriction that no two vectors are $k$-crossing can have infinite size (e.g.\ $\{(k,-k)\colon k\in\ZZ\}$ for $w=2$).
Similarly, there are infinite $k$-crossing-free families of vectors in $\ZZ^w$ which are not antichains (e.g.\ $\{(k,k)\colon k\in\ZZ\}$ for $w=2$).

Determining the value of $f(k,w)$ is the main focus of this paper.
The following striking simple conjecture was formulated in 2010 and never published, so we state it here with the kind permission of its authors.

\begin{conjecture}[Felsner, Krawczyk, Micek]
\label{conj:main}
For all\/ $k,w\geq 1$, we have
\[\text{$f(k,w)=k^{w-1}$.}\]
\end{conjecture}

At first, it is even not clear whether $f(k,w)$ is bounded for all $k$ and $w$.
We prove the conjecture for $1\leq w\leq 3$ and provide lower (matching the conjectured value) and upper bounds on $f(k,w)$ for $w\geq 4$.
Still, we are unable to resolve the conjecture in full generality.

\begin{theorem}
\label{thm:w=3}
For\/ $1\leq w\leq 3$ and\/ $k\geq 1$, we have
\[\text{$f(k,w)=k^{w-1}$.}\]
\end{theorem}

\begin{theorem}
\label{thm:bounds}
For\/ $w\geq 4$ and\/ $k\geq 1$, we have
\[\text{$k^{w-1}\leq f(k,w)\leq\min\{k^w-k^2(k-1)^{w-2},\lceil\tfrac{w}{3}\rceil k^{w-1}\}$.}\]
\end{theorem}

The remainder of this paper is organized as follows.
We start, in the next section, by a brief discussion of problems in partially ordered sets that initiated this research.
Section \ref{s:bounds} is devoted to the proof of Theorem \ref{thm:bounds} and the lower bound of Theorem~\ref{thm:w=3}.
The upper bound of Theorem \ref{thm:w=3} is proved in Section~\ref{s:proof-w=3}.
In Section \ref{s:generalization}, we propose another conjecture, which is at first glance more general but in fact equivalent to Conjecture \ref{conj:main}.
Concluding in Section \ref{s:examples}, we provide examples of families witnessing $f(k,w)\geq k^{w-1}$ with a discussion why the full resolution of the conjecture seems to be difficult.
We also present a proof of the conjecture for families of vectors with a single coordinate differentiating all vectors in the family and another argument for ranked families of vectors, that is, families in which the coordinates of every vector sum up to the same value.

\section{Background motivation}

Let $\cgM(P)$ denote the family of all maximum antichains (that is, antichains of maximum size) in a finite poset $P$.
The family $\cgM(P)$ is partially ordered by setting $A\leq B$ when for every $a\in A$ there is $b\in B$ with $a\leq b$ in $P$, or equivalently, when for every $b\in B$ there is $a\in A$ with $a\leq b$ in $P$.
The family $\cgM(P)$ equipped with this partial order forms a distributive lattice \cite{Dil60}, and every finite distributive lattice is isomorphic to $\cgM(P)$ for some poset $P$ \cite{Koh83}.
In the following, we are concerned with the order structure of $\cgM(P)$, in particular its width.

For positive integers $k_1,\ldots,k_n$, let $\boldsymbol{k_1}+\cdots+\boldsymbol{k_n}$ denote the poset consisting of $n$ pairwise disjoint chains of sizes $k_1,\ldots,k_n$ with no comparabilities between points in distinct chains.
For a positive integer $k$, let $\cgP(k)$ denote the class of posets containing no subposet isomorphic to $\boldsymbol{k}+\boldsymbol{k}$.
The posets in $\cgP(1)$ are just the chains, while $\cgP(2)$ is exactly the class of interval orders \cite{Fis70,Mir72}.
For positive integers $k$ and $w$, let $\cgP(k,w)$ denote the subclass of $\cgP(k)$ consisting of posets of width at most $w$.

Recently, several results in combinatorics of posets showed that problems that are difficult or even impossible to deal with for all posets of bounded width become much easier when only posets from $\cgP(k,w)$ are considered.
This includes the on-line chain partitioning problem \cite{BKS10,DJW12,JM11} and the on-line dimension problem \cite{FKT13}.

It is easy to see that the width of $\cgM(\boldsymbol{k}+\boldsymbol{k})$ is $k$, and it follows from Sperner's theorem \cite{Spe28} that the width of $\cgM(\underbrace{\boldsymbol{2}+\cdots+\boldsymbol{2}}_w)$ is $\binom{w}{\lfloor w/2\rfloor}$.
However, it turns out that the width of $\cgM(P)$ can be bounded by a constant when the width of $P$ is bounded and the size of a $\boldsymbol{k}+\boldsymbol{k}$ type structure in $P$ is bounded as well.

\begin{proposition}
\label{prop:relation}
For\/ $k,w\geq 1$ and\/ $P\in\cgP(k+1,w)$, the width of\/ $\cgM(P)$ is at most\/ $f(k,w)$.
\end{proposition}

\begin{proof}
By Dilworth's theorem \cite{Dil50}, $P$ can be covered with $w$ chains $C_1,\ldots,C_w$.
Each of them intersects each antichain $A\in\cgM(P)$.
Enumerate the elements of each chain $C_i$ as $c_{i,1},\ldots,c_{i,|C_i|}$ according to their order in the chain.
For an antichain $A\in\cgM(P)$, define a vector $A'\in\ZZ^w$ so that $A'[i]$ is the height in $C_i$ of the element common to both $A$ and $C_i$, that is, $A\cap C_i=\{c_{i,A'[i]}\}$ for $1\leq i\leq w$.
Clearly, for every antichain $\cgA\subset\cgM(P)$, the family of vectors $\cgA'=\{A'\colon A\in\cgA\}$ is an antichain in $\ZZ^w$.
Moreover, no two vectors in $\cgA'$ are $k$-crossing: if $A'[i]-B'[i]=k_1\geq k$ and $B'[j]-A'[j]=k_2\geq k$ for some $A,B\in\cgA$ and $1\leq i,j\leq w$, then the elements $c_{i,B'[i]},\ldots,c_{i,A'[i]}$ and $c_{j,A'[j]},\ldots,c_{j,B'[j]}$ induce a subposet of $P$ isomorphic to $(\boldsymbol{k_1+1})+(\boldsymbol{k_2+1})$, which contradicts the assumption that $P\in\cgP(k+1,w)$.
Therefore, we have $|\cgA|=|\cgA'|\leq f(k,w)$.
\end{proof}

\begin{conjecture}[Felsner, Krawczyk, Micek]
\label{conj:poset}
Let\/ $k$ and\/ $w$ be positive integers with\/ $k\geq 2$.
The maximum width of\/ $\cgM(P)$ for a poset\/ $P\in\cgP(k,w)$ is\/ $(k-1)^{w-1}$.
\end{conjecture}

This conjecture was made prior to the formulation of Conjecture \ref{conj:main}.
It follows from Proposition \ref{prop:relation} that Conjecture \ref{conj:main} implies Conjecture \ref{conj:poset}.
In particular, in view of Theorem \ref{thm:w=3}, Conjecture \ref{conj:poset} is true for $w\leq 3$.
In the special case $k=2$, since $\cgP(2,w)$ is the class of interval orders of width $w$, Conjecture \ref{conj:poset} states the well-known fact that the maximum antichains in an interval order form a chain.
Moreover, we are able to prove that Conjecture \ref{conj:poset} for $k=3$ and Conjecture \ref{conj:main} for $k=2$ are equivalent.

\section{General bounds}
\label{s:bounds}

The purpose of this section is to give the proof of Theorem \ref{thm:bounds} and the lower bound of Theorem \ref{thm:w=3}, namely, that we have
\begin{alignat*}{2}
f(k,w)&\geq k^{w-1}&&\quad\text{for $k,w\geq 1$},\\
f(k,w)&\leq\min\{k^w-k^2(k-1)^{w-2},\lceil\tfrac{w}{3}\rceil k^{w-1}\}&&\quad\text{for $w\geq 4$ and $k\geq 1$}.
\end{alignat*}

Note that $f(k,1)=k^0=1$ for every $k\geq 1$, as all antichains in $\ZZ^1$ are of size~$1$.
Also, $f(1,w)=1^{w-1}=1$ for every $w\geq 1$, as in this case every pair of distinct vectors is required to be simultaneously $1$-crossing and non-$1$-crossing.

For the lower bound, observe that the following family is a $k$-crossing-free antichain in $\ZZ^w$ and has size $k^{w-1}$:
\[
\{A\in\ZZ^w\colon\text{$0\leq A[i]\leq k-1$ for $1\leq i\leq w-1$, and $A[1]+\cdots+A[w]=0$}\}.
\]

For the upper bound, we start by an easy argument that yields the bound of $k^w$.
Let $\cgA$ be a $k$-crossing-free antichain in $\ZZ^w$.
For each vector $A\in\cgA$, let $\sigma(A)$ be the vector from $\{0,\ldots,k-1\}^w$ such that $A[i]\equiv\sigma(A)[i]\pmod{k}$ for $1\leq i\leq w$.
If $\sigma(A)=\sigma(B)$ for distinct vectors $A,B\in\cgA$, then any two coordinates $i$ and $j$ such that $A[i]>B[i]$ and $B[j]>A[j]$ (which must exist, as $\cgA$ is an antichain) witness that $A$ and $B$ are $k$-crossing.
It follows that $\sigma$ is an injection.
Since the size of the range of $\sigma$ is $k^w$, we have $|\cgA|\leq k^w$.

We obtain better upper bounds using the following recursive formula.

\begin{claim}\label{cla:recursive}
For\/ $w\geq 2$ and\/ $k\geq 1$, we have
\begin{align*}
&\text{$f(k,w)\leq k^{w-v}f(k,v)+k^vf(k,w-v)$,\quad for $1\leq v<w$,}\\
&\text{$f(k,w)\leq k^{w-1}+(k-1)f(k,w-1)$.}
\end{align*}
\end{claim}

\begin{proof}
Let $\cgA$ be a $k$-crossing-free antichain in $\ZZ^w$.
When $A\in\cgA$ and $1\leq i\leq j\leq w$, we will use $A[i,\ldots,j]$ as a convenient notation for the vector $(A[i],\ldots,A[j])$ in $\ZZ^{j-i+1}$.

Fix a residue class $(r_1,\ldots,r_{w-v})\in\{0,\ldots,k-1\}^{w-v}$, and consider the family $\cgA'$ of all vectors $A\in\cgA$ such that $A[i]\equiv r_i\pmod{k}$ for $1\leq i\leq w-v$.
For any distinct vectors $A,B\in\cgA'$, we have $A[w-v+1,\ldots,w]\neq B[w-v+1,\ldots,w]$, as otherwise $A$ and $B$ would be $k$-crossing.
Let $\cgA''=\{A[w-v+1,\ldots,w]\colon A\in\cgA'\}$.
The maximal vectors in $\cgA''$ form a $k$-crossing-free antichain, so there are at most $f(k,v)$ of them.
Color the vectors $A\in\cgA'$ such that $A[w-v+1,\ldots,w]$ is maximal in $\cgA''$ \emph{red} and the remaining vectors in $\cgA'$ \emph{blue}.
Hence there are at most $f(k,v)$ red vectors $\cgA'$ for the fixed residue class and at most $k^{w-v}f(k,v)$ red vectors in $\cgA$ altogether.

Now, fix a residue class $(r_{w-v+1},\ldots,r_w)\in\{0,\ldots,k-1\}^v$, and consider the family $\cgA'$ of all blue vectors $A\in\cgA$ such that $A[i]\equiv r_i\pmod{k}$ for $w-v+1\leq i\leq w$.
For any distinct vectors $A,B\in\cgA'$, we have $A[1,\ldots,w-v]\neq B[1,\ldots,w-v]$, as otherwise $A$ and $B$ would be $k$-crossing.
Let $\cgA''=\{A[1,\ldots,w-v]\colon A\in\cgA'\}$.
We show that $\cgA''$ is an antichain.
Suppose to the contrary that there are two vectors $A,B\in\cgA'$ such that $A[1,\ldots,w-v]<B[1,\ldots,w-v]$.
The vectors $A[w-v+1,\ldots,w]$ and $B[w-v+1,\ldots,w]$ are distinct, as otherwise we would have $A<B$, and comparable in $\ZZ^v$, as otherwise $A$ and $B$ would be $k$-crossing.
Hence we have $A[w-v+1,\ldots,w]>B[w-v+1,\ldots,w]$.
In particular, there is a coordinate $j\in\{w-v+1,\ldots,w\}$ such that $A[j]>B[j]$, which implies $A[j]-B[j]\geq k$.
By the definition of the coloring, there is a red vector $A'\in\cgA$ such that
\begin{enumeratei}
\item $A'[i]\equiv A[i]\pmod{k}$ for $1\leq i\leq w-v$, and
\item\label{item:recursive-ii} $A'[w-v+1,\ldots,w]>A[w-v+1,\ldots,w]$.
\end{enumeratei}
There is a coordinate $i\in\{1,\ldots,w-v\}$ such that $A'[i]<A[i]$, as otherwise we would have $A'>A$.
This implies $A[i]-A'[i]\geq k$.
This is a contradiction: we have $B[i]-A'[i]\geq A[i]-A'[i]\geq k$ and $A'[j]-B[j]\geq A[j]-B[j]\geq k$, so $A'$ and $B$ are $k$-crossing.
We have thus shown that $\cgA''$ is indeed an antichain.
Since $\cgA''$ is $k$-crossing-free, we have $|\cgA''|\leq f(k,w-v)$.
Hence $|\cgA'|\leq f(k,w-v)$ for the fixed residue class, and there are at most $k^vf(k,w-v)$ blue vectors in $\cgA$ altogether.

We conclude that the total number of red and blue vectors in $\cgA$ is at most $k^{w-v}f(k,v)+k^vf(k,w-v)$, as is required for the first inequality.
For the second one, if $v=1$, then it is enough to consider residue classes of $A[w]$ modulo $k-1$ instead of $k$ in the second part of the argument.
This is because $A'[w]-B[w]\geq k$ will follow from $A[w]-B[w]\geq k-1$ and the strict inequality $A'[w]>A[w]$ (see \ref{item:recursive-ii} above).
\end{proof}

From the second inequality of Claim \ref{cla:recursive} and the fact that $f(k,1)=1$, it follows that $f(k,w)\leq k^w-(k-1)^w$.
This bound is better than both $k^w$ and $wk^{w-1}$.
With the equality $f(k,3)=k^2$ of Theorem \ref{thm:w=3}, we get an even better bound
\[f(k,w)\leq k^w-k^2(k-1)^{w-2}\quad\text{for $w\geq 2$ and $k\geq 1$}.\]
The first inequality of Claim \ref{cla:recursive} applied recursively with $v=3$ and $f(k,3)=k^2$ give an upper bound
\[f(k,w)\leq\lceil\tfrac{w}{3}\rceil k^{w-1}\quad\text{for $w\geq 3$ and $k\geq 1$}.\]

\section{The case \texorpdfstring{$w\leq 3$}{w<=3}}
\label{s:proof-w=3}

In this section, we prove Theorem \ref{thm:w=3}, namely, that we have
\[f(k,w)=k^{w-1}\quad\text{for $1\leq w\leq 3$ and $k\geq 1$}.\]

As explained at the beginning of the previous section, the equality holds for $w=1$ or $k=1$.
Therefore, for the rest of this section, we assume that $2\leq w\leq 3$ and $k\geq 2$.
We only need to show that $f(k,w)\leq k^{w-1}$, as the converse inequality is proved in the previous section.
We start by the following easy proposition, stated for emphasis.

\begin{proposition}\label{prop:fix}
Let\/ $w\geq 2$, and let\/ $\cgA$ be an antichain in\/ $\ZZ^w$.
If\/ $S\subset\{1,\ldots,w\}$, $|S|=w-2$, and\/ $A[i]=B[i]$ for any\/ $A,B\in\cgA$ and every\/ $i\in S$, then the two remaining coordinates\/ $j,j'\in\{1,\ldots,w\}\setminus S$ determine two linear orders on\/ $\cgA$, one dual to the other.
That is, if we set\/ $n=|\cgA|$, then there is a labeling\/ $A_1,\ldots,A_n$ of the vectors in\/ $\cgA$ such that
\[\text{$A_1[j]<\cdots<A_n[j]$\quad and\quad $A_1[j']>\cdots>A_n[j']$.}\]
In particular, $A_1$ and\/ $A_n$ are\/ $(n-1)$-crossing.
\end{proposition}

It follows immediately from Proposition \ref{prop:fix} that $f(k,2)\leq k$.
Therefore, for the remainder of the argument, we fix $w=3$ and show that $f(k,3)\leq k^2$ for $k\geq 2$.

We say that a $k$-crossing-free antichain $\cgA$ in $\ZZ^w$ is \emph{compressed} on the $i$th coordinate when $A[i]\geq 0$ for all $A\in\cgA$ and the quantity $\sum_{A\in\cgA}A[i]$ is minimized over all $k$-crossing-free antichains of the same size.
Let $\cgA$ be a $k$-crossing-free antichain in $\ZZ^3$ compressed on the third coordinate.
It follows that $Q_3=\{A[3]\colon A\in\cgA\}$ is an interval of non-negative integers starting from $0$.
By Proposition \ref{prop:fix}, the subfamily of $\cgA$ consisting of all vectors $A$ with $A[3]=s$ has size at most $k$ for any $s\geq 0$.
We conclude that $|\cgA|\leq k^2$ if $|Q_3|\leq k$.
Thus, for the remainder of the argument, we assume $|Q_3|>k$.

Now, we use coordinate $3$ to define a directed graph $D$ whose vertices are the vectors in $\cgA$.
The edges in $D$ are of two types:\ \emph{short} and \emph{long}.
\begin{enumeratei}
\item $D$ has a short edge from $A$ to $B$ when $A[3]-B[3]=1$ and $A[i]\leq B[i]$ for $i\in\{1,2\}$.
\item $D$ has a long edge from $A$ to $B$ when $B[3]-A[3]=k-1$ and there is a coordinate $i\in\{1,2\}$ for which $A[i]-B[i]\geq k$.
\end{enumeratei}

\begin{claim}
For every\/ $A\in\cgA$, there is a path\/ $(A_0,\ldots,A_p)$ in\/ $D$ with\/ $A_0=A$ and\/ $A_p[3]=0$.
\end{claim}

\begin{proof}
The statement is trivial for $A\in\cgA$ with $A[3]=0$.
Suppose the conclusion of the claim is false for some vector $A\in\cgA$ with $A[3]>0$.
Let $\cgB$ denote the subfamily of $\cgA$ consisting of $A$ and the vectors $B$ in $\cgA$ for which there is a directed path from $A$ to $B$ in $D$.
Decrease coordinate $3$ of each vector in $\cgB$ by $1$, thus obtaining a family $\cgB'$.
The family $\cgA'=(\cgA\setminus\cgB)\cup\cgB'$ has the same size as $\cgA$, is an antichain, contains no two $k$-crossing vectors, uses only non-negative coordinates, and satisfies $\sum_{A\in\cgA'}A[i]<\sum_{A\in\cgA}A[i]$.
This contradicts the choice of $\cgA$ and completes the proof of the claim.
\end{proof}

\begin{claim}
\label{cla:2}
For every\/ $A\in\cgA$ with\/ $A[3]\geq k$, there is a path\/ $(U_0,\ldots,U_k)$ in\/ $D$ such that\/ $U_0[3]=A[3]$ and\/ $(U_m,U_{m+1})$ is a short edge in\/ $D$ for\/ $0\leq m\leq k-1$.
\end{claim}

\begin{proof}
Fix $A\in\cgA$ with $A[3]\geq k$.
For each $U\in\cgA$ with $U[3]=A[3]$, consider the length of a shortest path $P=(U_0,\ldots,U_p)$ in $D$ from $U$ to a vertex $U_p$ with $U_p[3]=0$.
Of all such $U$ and $U_p$, take those for which the length $p$ of the path $P$ is minimized.
We show that the first $k+1$ vectors on the chosen path satisfy the requirements of the claim.
Suppose to the contrary that there is $m$ with $0\leq m\leq k-1$ for which the edge $(U_m,U_{m+1})$ is long.
Then $U_{m+1}[3]\geq A[3]$, and it follows that there is an integer $n$ with $m+1\leq n<p$ for which $U_n[3]=A[3]$.
This contradicts the choice of $P$ and completes the proof of the claim.
\end{proof}

In view of Claim \ref{cla:2}, it is natural to refer to a path $\cgP=(U_0,\ldots,U_p)$ in $D$ as a \emph{short path} when all edges on $\cgP$ are short.
Also, we say that the short edge $(U,V)$ from $D$ is \emph{expanded in coordinate\/ $i$} when $V[i]>U[i]$.
Clearly, if $(U,V)$ is a short edge in $D$, then it is expanded in one or both of coordinates $1$ and $2$ (as $U$ and $V$ are $1$-crossing).

Let $\cgA_s=\{A\in\cgA\colon A[3]\equiv s\pmod{k}\}$ for $0\leq s\leq k-1$.
To complete the proof, we show that $|\cgA_s|\leq k$ for $0\leq s\leq k-1$.
Thus, for the remainder of the argument, we fix an integer $s$ with $0\leq s\leq k-1$.
Let $r$ be the largest integer for which there is a vector $A\in\cgA$ with $A[3]=s+(r-1)k$, and let $\cgA_s=\cgB_1\cup\cdots\cup\cgB_r$ be the natural partition of $\cgA_s$ such that $A[3]=s+(j-1)k$ for each $A\in\cgB_j$.
We can assume that $r\geq 2$, as otherwise the conclusion that $|\cgA_s|\leq k$ follows from Proposition~\ref{prop:fix}.

For $1\leq j\leq r$, we refer to $\cgB_j$ as \textit{level} $j$ of $\cgA_s$.
Also, for $1\leq j\leq r-1$, we apply Claim \ref{cla:2} and choose a short path $\cgP_j$ of $k+1$ vectors starting at a vector $X_{j+1}\in\cgB_{j+1}$ and ending at a vector $Y_j\in\cgB_j$.

\begin{claim}
For\/ $2\leq j\leq r-1$, we have\/ $X_j\neq Y_j$.
\end{claim}

\begin{proof}
Suppose to the contrary that for some $j$ with $2\leq j\leq r-1$ we have $X_j=Y_j$.
The ending point of the short path $\cgP_j$ is the same as the
starting point of the short path $\cgP_{j-1}$.
It follows that the union of these two paths is a short path of $2k+1$ vectors starting at the vector $X_{j+1}$ and ending at the vector $Y_{j-1}$.
Denote the vectors on this path by $\cgP=(U_0,\ldots,U_{2k})$, where $U_0=X_{j+1}$ and $U_{2k}=Y_{j-1}$.
We have 
\[U_0[i]\leq\cdots\leq U_{2k}[i]\quad\text{for $i\in\{1,2\}$}.\]
Furthermore, for $0\leq m\leq 2k-1$, the short edge $(U_m, U_{m+1})$ is expanded in some coordinate $i\in\{1,2\}$.
Since there are $2k$ short edges on $\cgP$, at least $k$ of them are expanded in some coordinate $i\in\{1,2\}$.
It follows that $U_{2k}[i]-U_0[i]\geq k$.
Since $U_0[3]-U_{2k}[3]=2k$, we conclude that $U_0$ and $U_{2k}$ are $k$-crossing.
This contradiction completes the proof of the claim.
\end{proof}

Let $j$ be an integer with $1\leq j\leq r$.
Since $A[3]=s+(j-1)k$ for all $A\in\cgB_j$, we know from Proposition \ref{prop:fix} that (a)~each of the first two coordinates determines a linear order on $\cgB_j$, and (b)~these two linear orders are dual.
In particular, if $2\leq j\leq r-1$, then there is a unique $i\in\{1,2\}$ for which $X_j[i]>Y_j[i]$.

Now let $i\in\{1,2\}$.
An interval $B=[p,t]$ of consecutive integers from $[1,r-1]$ is called a \textit{block of type\/ $i$} when the following conditions are satisfied:
\begin{enumeratei}
\item $p=1$ or $X_p[i]<Y_p[i]$;
\item $X_j[i]>Y_j[i]$ for all $j\in(p,t]$;
\item $t=r-1$ or $X_{t+1}[i]<Y_{t+1}[i]$.
\end{enumeratei}
The blocks of type $i$ form a partition of the integer interval $[1,r-1]$.
In particular, every $j\in[1,r-1]$ belongs to two blocks, one of each type.
Moreover, for every $j\in[1,r-2]$, there is a unique $i$ such that $j$ and $j+1$ belong together to a block of type $i$.
This implies that there are exactly $r$ blocks altogether.
When $r=2$, the singleton set $\{1\}$ is a block of both types, as the three conditions listed above are satisfied vacuously, and it is counted twice.

Choose $j$ with $1\leq j\leq r-1$.
Let $\cgP_j=(U_0,\ldots,U_k)$.
For $i\in\{1,2\}$, let $B_i$ be the block of type $i$ containing $j$, that is, $B_i=[p_i,t_i]$ with $p_i\leq j\leq t_i$.
When a short edge $(U_m,U_{m+1})$ with $0\leq m\leq k-1$ is expanded in coordinate $i$, we say that $(U_m,U_{m+1})$ is \emph{expanded in\/ $B_i$}.
Each of the short edges $(U_m,U_{m+1})$, for $0\leq m\leq k-1$, is expanded in at least one of $B_1$ and~$B_2$.

Now, choose $j$ with $1\leq j\leq r$.
Let $U$ and $V$ be distinct vectors in $\cgB_j$ that occur consecutively in the two linear orders induced by coordinates $1$ and $2$.
We say that the pair $(U,V)$ \emph{contributes a space} to a block $B=[p,t]$ of type $i$ when one of the following three conditions is satisfied:
\begin{enumeratei}
\item $j=p$ and $U[i]>V[i]\geq Y_j[i]$;
\item $p+1\leq j\leq t$ and $X_j[i]\geq U[i]>V[i]\geq Y_j[i]$;
\item $j=t+1$ and $X_j[i]\geq U[i]>V[i]$.
\end{enumeratei}

\begin{claim}
\label{cla:4}
Let\/ $j$ be an integer with\/ $1\leq j\leq r$.
If\/ $U,V\in\cgB_j$ are consecutive in the linear orders on\/ $\cgB_j$ determined by coordinates\/ $1$ and\/ $2$, then exactly one of\/ $(U,V)$ and\/ $(V,U)$ contributes a space to a block and that block is unique.
\end{claim}

\begin{proof}
Assume without loss of generality that $U[1]>V[1]$ and $V[2]>U[2]$.

Suppose first that $j=1$.
If $U[1]>V[1]\geq Y_1[1]$, then $(U,V)$ contributes a space to the block of type $1$ containing $1$.
Otherwise, we have $V[2]>U[2]\geq Y_1[2]$ and $(V,U)$ contributes a space to the block of type $2$ containing~$1$.

The proof for the case $j=r$ is similar.
If $X_r[1]\geq U[1]>V[1]$, then $(U,V)$ contributes a space to the block of type $1$ containing $r-1$.
Otherwise, $(V,U)$ contributes a space to the block of type $2$ containing $r-1$.

Now, suppose $2\leq j\leq r-1$.
There is a unique block $B$ containing both $j-1$ and $j$.
Assume without loss of generality that $B$ is a block of type $1$.
If $X_j[1]\geq U[1]>V[1]\geq Y_j[1]$, then $(U,V)$ contributes a space to $B$.
If $U[1]>V[1]\geq X_j[1]$, then $(V,U)$ contributes a space to the block of type $2$ containing $j-1$.
Finally, if $Y_j[1]\geq U[1]>V[1]$, then $(V,U)$ contributes a space to the block of type $2$ that contains~$j$.
\end{proof}

\begin{claim}
\label{cla:5}
For every block\/ $B$, the total number of pairs that contribute a space to\/ $B$ and short edges that expand in\/ $B$ is at most\/ $k-1$.
\end{claim}

\begin{proof}
Let $B=[p,t]$ be a block of type $i$.
For $p\leq j\leq t+1$, let $V_j^0,\ldots,V_j^{n_j}$ be the vectors $V$ from $\cgB_j$ such that
\begin{enumeratei}
\item $V[i]\geq X_j[i]$ when $j\geq p+1$,
\item $V[i]\leq Y_j[i]$ when $j\leq t$.
\end{enumeratei}
Assume further that $V_j^0,\ldots,V_j^{n_j}$ are ordered so that $V_j^0[i]>\cdots>V_j^{n_j}[i]$.
Thus $V_j^{n_j}=X_j$ for $p+1\leq j\leq t+1$, and $V_j^{n_j}=Y_j$ for $p\leq j\leq t$.
Clearly, the pairs $(V_j^m,V_j^{m+1})$ with $p\leq j\leq t+1$ and $0\leq m\leq n_j-1$ are exactly the pairs that contribute a space to $B$.

For $p\leq j\leq t$, let $\cgP_j=(U_j^0,\ldots,U_j^k)$.
Thus $U_j^0=X_{j+1}$, $U_j^k=Y_j$, and $U_j^k[i]\geq\cdots\geq U_j^0[i]$.
Clearly, the short edges $(U_j^m,U_j^{m+1})$ with $p\leq j\leq t$, $0\leq m\leq k-1$, and $U_j^{m+1}[i]>U_j^m[i]$ are exactly the short edges that expand in $B$.

To conclude, since we have
\begin{alignat*}{8}
&V_p^0&&[i]>\cdots>V_p^{n_p}&&[i]=Y_p&&[i]=U_p^k&&[i]\geq\cdots&&\geq U_p^0&&[i]=X_{p+1}&&[i]\\
={}&V_{p+1}^0&&[i]>\cdots>V_{p+1}^{n_{p+1}}&&[i]=Y_{p+1}&&[i]=U_{p+1}^k&&[i]\geq\smash[b]{\raisebox{-11pt}{$\ddots$}}\\
&&&&&&&&&&&\geq U_t^0&&[i]=X_{t+1}&&[i]\\
={}&V_{t+1}^0&&[i]>\cdots>V_{t+1}^{n_{t+1}}&&[i],
\end{alignat*}
the total number of pairs that contribute a space to $B$ and short edges that expand in $B$ is at most $V_p^0[i]-V_{t+1}^{n_{t+1}}[i]$.
Since $V_{t+1}^{n_{t+1}}[3]-V_p^0[3]=(t-p+1)k$, we have $V_p^0[i]-V_{t+1}^{n_{t+1}}[i]\leq k-1$, as otherwise $V_p^0$ and $V_{t+1}^{n_{t+1}}$ would be $k$-crossing.
\end{proof}

We are now ready to assemble this series of claims and complete the proof that $|\cgA_s|\leq k$.
For $1\leq j\leq r$, let $b_j=|\cgB_j|$.
Thus $|\cgA_s|=b_1+\cdots+b_r$.
By Claim \ref{cla:4}, there are $b_j-1$ ordered pairs of elements from $\cgB_j$ that occur consecutively in the linear orders on $\cgB_j$ induced by coordinates $1$ and $2$ and each contributes a space to one of the $r$ blocks.
Also, each of the $(r-1)k$ short edges on the paths $\cgP_1,\ldots,\cgP_{r-1}$ is expanded in at least one block.
Thus, by Claim \ref{cla:5}, we have
\[\sum_{j=1}^r(b_j-1)+(r-1)k\leq r(k-1).\]
On the other hand, we have
\[\sum_{j=1}^r(b_j-1)=|\cgA_s|-r.\]
Composing the two, we obtain $|\cgA_s|\leq k$, which completes the proof
of Theorem \ref{thm:w=3}.

\section{Generalization}
\label{s:generalization}

For $w\geq 1$ and $1\leq k_1\leq\cdots\leq k_w$, we say that vectors $A$ and $B$ from $\ZZ^w$ are \emph{$(k_1,\ldots,k_w)$-crossing} when there are two coordinates $i$ and $j$ for which $A[i]-B[i]\geq k_i$ and $B[j]-A[j]\geq k_j$.
Let $f(k_1,\ldots,k_w;w)$ denote the maximum size of a $(k_1,\ldots,k_w)$-crossing-free antichain of vectors in $\ZZ^w$.
Thus $f(k,w)=f(k,\ldots,k;w)$.

\begin{proposition}
\label{prop:genbounds}
For\/ $w\geq 1$ and\/ $k_1,\ldots,k_w\geq 1$, we have
\[\text{$k_2\cdots k_w\leq f(k_1,\ldots,k_w;w)\leq k_1\cdots k_w$.}\]
\end{proposition}

The proof of Proposition \ref{prop:genbounds} follows along the same lines as the proof of the inequalities $k^{w-1}\leq f(k,w)\leq k^w$ at the beginning of Section \ref{s:bounds}.
We propose a conjecture which seems to be more general but turns out to be equivalent to Conjecture~\ref{conj:main}.

\begin{conjecture}
\label{conj:general}
For\/ $w\geq 1$ and\/ $1\leq k_1\leq\cdots\leq k_w$, we have
\[\text{$f(k_1,\ldots,k_w;w)=k_2\cdots k_w$.}\]
\end{conjecture}

\begin{proposition}
\label{prop:equiv}
Conjectures \ref{conj:main} and \ref{conj:general} are equivalent.
\end{proposition}

\begin{proof}
Clearly, Conjecture \ref{conj:general} yields Conjecture \ref{conj:main}.
To prove the converse implication, we assume $f(k_1,\ldots,k_1;w)=k_1^{w-1}$ and prove $f(k_1,\ldots,k_w;w)=k_2\cdots k_w$.
Let $\cgA$ be a $(k_1,\ldots,k_w)$-crossing-free antichain in $\ZZ^w$.
For any selection of $k_1$-element subsets $I_2\subset\{0,\ldots,k_2-1\}$, \ldots, $I_w\subset\{0,\ldots,k_w-1\}$, consider the family $\cgA(I_2,\ldots,I_w)=\{A\in\cgA\colon A[i]\bmod k_i\in I_i$ for $2\leq i\leq w\}$.
Now, modify each $A\in\cgA(I_2,\ldots,I_w)$ to get a vector $A'$ so that if $A[j]=a_jk_j+r_j$, where $0\leq r_j<k_j$, and $\ell_j$ is the position of $r_j$ in the natural ordering of $I_j$, then $A'[j]=a_jk_1+\ell_j$.
Clearly, the resulting family $\cgA'$ of all the vectors $A'$ is a $k_1$-crossing-free antichain.
Thus $|\cgA(I_2,\ldots,I_w)|=|\cgA'|\leq k_1^{w-1}$.
Summing up over all selections of subsets $I_2,\ldots,I_w$, we obtain
\[
\tbinom{k_2-1}{k_1-1}\cdots\tbinom{k_w-1}{k_1-1}{|\cgA|}\leq\tbinom{k_2}{k_1}\cdots\tbinom{k_w}{k_1}k_1^{w-1},
\]
which implies $|\cgA|\leq k_2\cdots k_w$.
\end{proof}

Proposition \ref{prop:equiv} tells us that in some sense the most difficult case is when all $k_i$ are equal.
Surprisingly, for some values of $k_i$, we know the exact answer.
For instance,
\[f(k,k,2k,\ldots,2^{w-1}k;w)=k\cdot 2k\cdots 2^{w-1}k.\]
Namely, we show that
\[f(k,k,2k,\ldots,2^{w-1}k;w)\leq kf(2k,2k,\ldots,2^{w-1}k;w-1),\]
which together with $f(k,k;2)=k$ and Proposition \ref{prop:genbounds} gives the previous equality.
We write $A<_1^2B$ if $A[1]<B[1]$ and $A[2]>B[2]$.
Every maximum chain in this order has size at most $k$, as otherwise it would yield a $(k,k,2k,\ldots,2^{w-1}k)$-crossing.
Let $\cgA'$ be a family of vectors of a fixed height in the order $<_1^2$.
Now let $\phi(A)=(A[1]+A[2],A[3],\ldots,A[w])$ for $A\in\cgA'$.
The mapping $\phi$ is an injection, and $\phi(\cgA')$ is a $(2k,2k,\ldots,2^{w-1}k)$-crossing-free antichain in $\ZZ^{w-1}$.
This gives the required inequality.

\section{Extremal examples}
\label{s:examples}

Some classical extremal problems have elegant solutions due to the fact that all maximal structures are also maximum.
For example, the maximum number of edges in a planar graph is $3n-6$ when $n\geq 3$, because if $G$ is any planar graph containing a face that is not a triangle, then an edge can be added to $G$ while preserving planarity.

Other extremal problems can have many different maximal structures but essentially only one which is maximum.
An example of this is Tur\'an's theorem, which asserts that the maximum number of edges in a graph on $n$ vertices which does not contain a complete subgraph on $k+1$ vertices is the number of edges in the complete $k$-partite graph on $n$ vertices, where the part sizes are as balanced as possible.
Another example is Sperner's theorem, which asserts that the only maximum antichains in the lattice of all subsets of $\{1,\ldots,n\}$ are the ranks at levels $\lfloor n/2\rfloor$ and $\lceil n/2\rceil$.

It is our feeling that the extremal problem discussed in this paper is challenging because there are many different examples that we suspect to be extremal.
We already presented one example at the beginning of Section \ref{s:bounds}, and in this section we develop some others.

\subsection{Inductive construction}

Suppose that we have constructed a $k$-crossing-free antichain $\cgA$ in $\ZZ^w$, and suppose it is contained in $[0,c)^w$.
We are going to construct an antichain $\cgA'$ of size $k{|\cgA|}$ on $w+1$ coordinates.
Put $k$ disjoint copies of $\cgA$ one above another on coordinates $1,\ldots,w$, that is, the $i$th copy inside $[(i-1)c,ic)^w$, and set the coordinate $w+1$ to be $-i$ for all vectors in the $i$th copy.
This way we obtain a $k$ times larger $k$-crossing-free antichain in $\ZZ^{w+1}$.
If $|\cgA|=k^{w-1}$, then $|\cgA'|=k^w$.

\subsection{Lexicographic construction}

When $A\in\ZZ^w$, the \textit{rank} of $A$ is the sum $A[1]+\cdots+A[w]$.
Let $k,w\geq 2$.
We construct an antichain $\cgA$ in $\ZZ^w$ as follows.
First, consider the family $\cgF$ of all vectors in $\ZZ^w$ with $0\leq A[i]\leq k-1$ for $1\leq i\leq w$ and $\sum_{i=1}^w A[i]\equiv w(k-1)\pmod{k}$.
Clearly, there are $k^{w-1}$ vectors in $\cgF$.
For each $A\in\cgF$, there is a unique non-negative integer $m(A)$ such that
\[m(A)\cdot k+A[1]+\cdots+A[w]=w(k-1).\]
Let $n$ be the maximum value of $m(A)$ taken over all vectors $A\in\cgF$.
Then, let $\tau=(i_1,\ldots,i_n)$ be any sequence of integers from $\{1,\ldots,w\}$.
We modify $\cgF$ into an antichain $\cgA$ by the following rule.
If $A\in\cgF$, then we modify $A$ by increasing coordinate $i$ by $pk$, where $p$ is the number of times $i$ occurs at the first $m(A)$ positions of $\tau$.
Clearly, these modifications result in a family $\cgA$ consisting of $k^{w-1}$ vectors.
Furthermore, since each vector $A\in\cgA$ has rank $w(k-1)$, we know that $\cgA$ is an antichain.
Also, $\cgA$ is $k$-crossing-free.

The example presented at the beginning of Section \ref{s:bounds} with the $w$th coordinate of all vectors shifted up by $w(k-1)$ is the special case of this construction where $\tau$ is the constant sequence $(w,\ldots,w)$.

\subsection{Cyclic construction}
\label{subsec:cyclic}

Here, we fix $w=3$ and consider coordinates $\{1,2,3\}$ in the cyclic order.
Thus if $i=3$ then $i+1=1$, and if $i=1$ then $i-1=3$.
Let $k\geq 2$.
Consider the infinite family
\[
\cgF=\{A\in\ZZ^3\colon A[i+1]\leq A[i]+k\text{ and }A[i-1]\leq A[i]+k-1\text{ for }i\in\{1,2,3\}\}.
\]
Clearly, it contains no two $k$-crossing vectors.
If $k\equiv 0$ or $k\equiv 2\pmod{3}$, then the vectors in $\cgF$ of rank $2k-1$ form an antichain of size $k^2$.
If $k\equiv 1\pmod{3}$, then the vectors in $\cgF$ of rank $2k-2$ form an antichain of size $k^2$.
In both cases, there is a cyclic symmetry between all three coordinates.

When $k\equiv 1\pmod{3}$, the vectors in $\cgF$ of rank $2k-1$ form an antichain of size $k^2-1$ only.
Still, we can add the vector $(\frac{k-1}{3}+k,\frac{k-1}{3},\frac{k-1}{3})$ to obtain a $k$-crossing-free antichain of size $k^2$ at the price of losing the cyclic symmetry.

\subsection{Remarks on rank}

All the examples we have constructed so far are \emph{ranked} antichains, that is, they consist of vectors in $\ZZ^w$ all of which have the same rank.
Based on this observation, it would be tempting to try to reduce the entire problem to ranked antichains.
Indeed, we have the following proposition.

\begin{proposition}
\label{prop:rank}
For all\/ $k,w\geq 1$, the maximum size of a ranked\/ $k$-crossing-free antichain in\/ $\ZZ^w$ is\/ $k^{w-1}$.
\end{proposition}

\begin{proof}
We only need to prove that if $\cgA$ is a ranked $k$-crossing-free antichain in $\ZZ^w$, then $|\cgA|\leq k^{w-1}$.
We can assume as before that $k,w\geq 2$.
For each vector $A$ in $\cgA$, let $\sigma(A)$ denote the vector in $\{0,\ldots,k-1\}^{w-1}$ such that $A[i]\equiv\sigma(A)[i]\pmod{k}$ for $1\leq i\leq k-1$.
Clearly, $\sigma$ is an injection and its range has at most $k^{w-1}$ elements.
\end{proof}

However, we know examples of $k$-crossing-free antichains in $\ZZ^w$ of the conjectured extremal size $k^{w-1}$ that are intrinsically non-ranked.
For example, for $k=2$ and $w=4$, the following eight vectors form a non-ranked $2$-crossing-free antichain in~$\ZZ^4$:
\[
\begin{array}{llll}
(0,2,1,1), &(2,1,0,1), &(1,0,2,1), &(1,1,1,1),\\
(1,3,2,0), &(3,2,1,0), &(2,1,3,0), &(2,2,2,0).
\end{array}
\]
The first four of the vectors above have rank $4$, while the last four have rank~$6$.
Moreover, this antichain is compressed on each of the four coordinates.
More generally, any family obtained by the cyclic construction (\ref{subsec:cyclic}) can be extended to $w=4$ in an analogous manner.

\subsection{Remarks on the size of the largest coordinate}

\begin{proposition}
\label{prop:rbigcoord}
Let\/ $k$ and\/ $w$ be positive integers.
Suppose that\/ $\cgA$ is a\/ $k$-crossing-free antichain in\/ $\ZZ^w$, and suppose that there is a coordinate\/ $j$ on which all vectors are different.
Then\/ $|\cgA|\leq k^{w-1}$.
\end{proposition}

\begin{proof}
Suppose all vectors differ on the first coordinate.
For $2\leq i\leq w$, we define an order $<_i$ on $\cgA$ as follows.
We put $A<_iB$ if $A[1]<B[1]$ and $A[i]>B[i]$.
The maximum size of a chain in this order is at most $k$, as otherwise $\cgA$ would have two $k$-crossing vectors.
Let $\phi(A)\in\{1,\ldots,k\}^{w-1}$ be the vector of heights of $A$ in orders ${<_2},\ldots,{<_w}$.
Clearly, if $A,B\in\cgA$ are such that $A[1]<B[1]$, then for some coordinate $i$ we have $A[i]>B[i]$, and thus the heights of $A$ and $B$ in $<_i$ are different.
This shows that the mapping $\phi\colon A\to\{1,\ldots,k\}^{w-1}$ is injective.
\end{proof}

It follows that in any $k$-crossing-free antichain in $\ZZ^w$ the number of different values attained on any coordinate is at most $k^{w-1}$.
Otherwise, a choice of representatives of the attained values would contradict Proposition~\ref{prop:rbigcoord}.

\subsection{Remarks on compression}

Careful analysis of the proof of the case $w=3$ shows that we do not really need a fully compressed coordinate.
We only use the following two properties:
\begin{enumerate}[label={($P_{\arabic*}$)}, noitemsep, topsep=2pt plus 2pt, leftmargin=*]
\item\label{item:prop1} For $1\leq j\leq r$, the set $\cgB_j=\{A\in\cgA\colon A[3]=s+(j-1)k\}$ is an antichain.
\item\label{item:prop2} For $1\leq j\leq r-1$, there is a short path from a vector $X_{j+1}$ in $\cgB_{j+1}$ to a vector $Y_j$ in $\cgB_j$.
\end{enumerate}

However, when $w=4$, this weaker notion of compression (with $A[3]$ replaced by $A[4]$) is not enough.
To see this, consider the union of the following families of vectors in $\ZZ^4$:
\begin{enumeratei}
\item The vectors for which $0\leq A[1],A[2]\leq k-1$, $A[3]\geq 2$, $A[1]+A[2]+A[3]=2k-2$, and $A[4]=k$.
\item All vectors of the form $(i,k-1-i,k+1,0)$ where $0\leq i\leq k-1$.
\item The vector $(k-1,k-1,k,0)$.
\item All vectors having rank $3k-2$ with $1\leq A[i]\leq k-1$ for $1\leq i\leq 4$.
\end{enumeratei}
It is easy to see that this family satisfies properties \ref{item:prop1} and \ref{item:prop2} but has more than $k^2$ vectors for which $A[4]\equiv 0\pmod{k}$.

\section*{Acknowledgments}

We are very grateful to Tomasz Krawczyk, who invented the problem, and Stefan Felsner for their significant contribution at the early stage of research.
We thank Dave Howard, Mitch Keller, Jakub Kozik, Ruidong Wang, Marcin Witkowski, and Stephen Young for their helpful comments and observations.
We also thank anonymous reviewers whose suggestions helped us to improve the quality of this paper.

\bibliographystyle{plain}
\bibliography{k-crossing}

\end{document}